\documentclass[11pt]{amsart}

\usepackage{a4wide,amsmath,amsfonts,amssymb,amsthm,mathrsfs}
\usepackage[hyperfootnotes=false]{hyperref}
\usepackage{xcolor}

\allowdisplaybreaks[2]

\numberwithin{equation}{section}

\newtheorem{theorem}{Theorem}[section]
\newtheorem{lemma}[theorem]{Lemma}

\newtheorem{remark}[theorem]{Remark}
\newtheorem{proposition}[theorem]{Proposition}
\newtheorem{definition}[theorem]{Definition}
\newtheorem{assumption}[theorem]{Assumption}

\newcommand{\dd}{\,\mathrm{d}}
\newcommand{\R}{\mathbb{R}}
\newcommand{\N}{\mathbb{N}}
\newcommand{\E}{\mathbb{E}}
\renewcommand{\P}{\mathbb{P}}

\title[On the existence of weak solutions to stochastic Volterra equations]{On the existence of weak solutions\\ to stochastic Volterra equations}

\author[Pr{\"o}mel]{David J. Pr{\"o}mel}
\address{David J. Pr{\"o}mel, University of Mannheim, Germany}
\email{proemel@uni-mannheim.de}

\author[Scheffels]{David Scheffels}
\address{David Scheffels, University of Mannheim, Germany}
\email{dscheffe@mail.uni-mannheim.de}

\date{\today}

\begin{document}
 
\begin{abstract}
  The existence of weak solutions is established for stochastic Volterra equations with time-inhomogeneous coefficients allowing for general kernels in the drift and convolutional or bounded kernels in the diffusion term. The presented approach is based on a newly formulated local martingale problem associated to stochastic Volterra equations.
\end{abstract}

\maketitle

\noindent \textbf{Key words:} local martingale problem, singular kernel, stochastic Volterra equation, non-Lipschitz coefficients, weak existence.

\noindent \textbf{MSC 2020 Classification:} 60H20, 45D05.



\section{Introduction}

We investigate the existence of weak solutions to stochastic Volterra equation (SVEs)
\begin{equation}\label{eq:intro}
  X_t = x_0(t)+ \int_0^t K_\mu(s,t)\mu(s,X_s)\dd s + \int_0^t K_\sigma(s,t) \sigma (s,X_s) \dd B_s, \quad t\in [0,T],
\end{equation}
where $x_0$ is a continuous function, $B$ is a Brownian motion, and the kernels $K_\mu, K_\sigma$ are measurable functions. The time-inhomogeneous coefficients $\mu,\sigma$ are only supposed to be continuous in space uniformly in time. In case of ordinary stochastic differential equations (SDEs), i.e. $K_\sigma=K_\mu=1$, the existence of weak solutions was first proven by Skorokhod~\cite{Skorohod1961} and can, nowadays, be found in different generality in standard textbooks like \cite{Stroock1979,Karatzas1991}.

A comprehensive study of weak solutions to stochastic Volterra equations was recently initiated by Abi Jaber, Cuchiero, Larsson and Pulido~\cite{AbiJaber2021}, see also~\cite{Mytnik2015}. The extension of the theory of weak solutions from ordinary stochastic differential equations to SVEs constitutes a natural generalization of the classical theory and is motivated by successful applications of SVEs with non-Lipschitz coefficients as volatility models in mathematical finance, see e.g. \cite{ElEuch2019,AbiJaberElEuch2019b}. Assuming that the kernels in the SVE~\eqref{eq:intro} are of convolution type, i.e. $K_\mu (s,t)= K_\sigma(s,t)=K(t-s)$ for some function $K\colon \R\to \R$, and that the coefficients $\mu,\sigma$ are continuous jointly in space-time, the existence of weak solutions was derived in~\cite{AbiJaber2021}, see also \cite{Mytnik2015,AbiJaber2019,AbiJaber2021b}. To that end, Abi Jaber et al. \cite{AbiJaber2021} introduces a local martingale problem associated to SVEs of convolutional type.

In the present work we establish a local martingale problem associated to general stochastic Volterra equations, see Definition~\ref{def:martproblem}, and show that its solvability is equivalent to the existence of a weak solution to the associated SVE, see Lemma~\ref{lem:equivalence}. Using this newly formulated Volterra local martingale problem, we obtain the existence of weak solutions to stochastic Volterra equations with time-inhomogeneous coefficients, that are not necessarily continuous in $t$, and allowing for general kernels in the drift and convolutional kernels as well as bounded general kernels in the diffusion term, see Theorem~\ref{thm: weak existence}. The presented approach can be considered, roughly speaking, as a generalization of Skorokhod's original construction to the more general case of SVEs, and is developed in a one-dimensional setting to keep the presentation fairly short without cumbersome notation. However, as for ordinary SDEs and for SVEs of convolutional type, all concepts and results are expected to extend to a multi-dimensional setting in a straightforward manner.

\medskip

\noindent \textbf{Organization of the paper:} In Section~\ref{sec:VMP} we introduce a local martingale problem associated to SVEs. The existence of weak solutions to SVEs is provided in Section~\ref{sec:construction}.

\medskip

\noindent\textbf{Acknowledgments:} D. Scheffels gratefully acknowledges financial support by the Research Training Group ``Statistical Modeling of Complex Systems'' (RTG 1953) funded by the German Science Foundation (DFG).

\section{Weak solutions and the Volterra local martingale problem}\label{sec:VMP}

For $T\in (0,\infty)$ we consider the one-dimensional stochastic Volterra equation
\begin{equation}\label{eq:SVE}
  X_t = x_0(t)+ \int_0^t K_\mu(s,t)\mu(s,X_s)\dd s + \int_0^t K_\sigma(s,t) \sigma (s,X_s) \dd B_s, \quad t\in [0,T],  
\end{equation}
where $x_0\colon [0,T]\to\R$ is a continuous function, $(B_t)_{t\in [0,T]}$ is a Brownian motion on a probability space $(\Omega,\mathcal{F},\mathbb{P})$, and the coefficients $\mu,\sigma\colon [0,T]\times\R\to\R $ and the kernels $K_\mu, K_\sigma\colon \Delta_T\to \R$ are measurable functions, using the notation $\Delta_T:=\lbrace (s,t)\in [0,T]\times [0,T]\colon \, 0\leq s\leq t\leq T \rbrace$. The integral $\int_0^t K_{\mu}(s,t)\mu(s,X_s)\dd s$ is defined as a Riemann--Stieltjes integral and $\int_0^t K_{\sigma}(s,t)\sigma(s,X_s)\dd B_s$ as an It{\^o} integral. Moreover, for $p\in [1,\infty)$ we write $L^p(\Omega\times [0,T])$ and $L^p([0,T])$ for the space of $p$-integrable functions on $\Omega\times [0,T]$ and on $[0,T]$, respectively.

\medskip

Analogous to the notion of weak solutions to ordinary stochastic differential equations (see e.g. \cite[Chapter~5.3, Definition~3.1]{Karatzas1991}, we make the following definition.

\begin{definition}\label{def:weak solution}
  A \textup{weak solution} to \eqref{eq:SVE} is a triple $(X,B)$, $(\Omega,\mathcal{F},\mathbb{P})$, $(\mathcal{F}_t)_{t\in [0,T]}$ such that
  \begin{enumerate}
    \item[(i)] $(\Omega,\mathcal{F},\mathbb{P})$ is a probability space, $(\mathcal{F}_t)_{t\in [0,T]}$ is a filtration of sub-$\sigma$-algebras of $\mathcal{F}$ satisfying the usual conditions,
    \item[(ii)] $X=(X_t)_{t\in [0,T]}\in L^1(\Omega~\times~ [0,T])$ is an $(\mathcal{F}_t)$-progressively measurable process, $B=(B_t)_{t\in [0,T]}$ is a Brownian motion w.r.t. $(\mathcal{F}_t)_{t\in [0,T]}$,
    \item[(iii)] $\int_0^t\big(|K_\mu(s,t)\mu(s,X_s)| + |K_\sigma(s,t)\sigma(s,X_s)|^2  \big)\dd s<\infty$ $\P$-a.s. for any $t\in[0,T]$, and
    \item[(iv)] \eqref{eq:SVE} holds for $(X,B)$ on $(\Omega,\mathcal{F},\mathbb{P})$, $\P$-a.s.
  \end{enumerate}
\end{definition}

Under suitable assumptions on the coefficients and kernels, the existence of weak solutions to the stochastic Volterra equation~\eqref{eq:SVE} can be equivalently formulated in terms of solutions to an associated local martingale problem, see Definition~\ref{def:martproblem} below. To that end, we make the following assumption.

\begin{assumption}\label{ass:short}
  Let $K_\mu, K_\sigma\colon \Delta_T\to \R$ be measurable functions with $K_\mu(\cdot,t) \in L^1([0,T])$ and $K_\sigma(\cdot,t)\in L^2([0,T])$ for every $t\in[0,T]$, and let $\mu,\sigma\colon [0,T]\times\R\to\R$ be measurable functions fulfilling the linear growth condition
  \begin{equation*}
    |\mu(t,x)|+|\sigma(t,x)| \leq C_{\mu,\sigma} (1 +|x|), \quad  t\in [0,T],\, x\in \R,
  \end{equation*}
  for some constant $C_{\mu,\sigma}>0$.
\end{assumption}

Let $ C^2(\R)$ be the space of twice continuously differentiable functions $f\colon\R\to \R$ and $C_0^2(\R)$ be the space of all $f\in C^2(\R)$ with compact support. For two stochastic processes $X=(X_t)_{t\in [0,T]}$ and $Z=(Z_t)_{t\in [0,T]}$ on a filtered probability space $(\Omega,\mathcal{F},(\mathcal{F}_t)_{t\in[0,T]},\P)$ satisfying the usual conditions, such that $X\in L^1(\Omega\times [0,T])$ is $(\mathcal{F}_t)$-progressively measurable and $Z$ is $(\mathcal{F}_t)$-adapted and continuous, we introduce the process $(\mathcal{M}_t^f)_{t\in[0,T]}$ by
\begin{equation}\label{def:M_f1}
  \mathcal{M}^f_t:= f(Z_t)-\int_0^t \mathcal{A}^f(s,X_s,Z_s)\dd s,\quad t\in [0,T],
\end{equation}
for $f\in C^2(\R)$, where
\begin{equation}\label{eq:operator A}
  \mathcal{A}^f\colon[0,T]\times\R\times \R\to \R \quad\text{with}\quad \mathcal{A}^f(t,x,z):=\mu(t,x)f^\prime(z)+\frac{1}{2}\sigma(t,x)^2f^{\prime\prime}(z).
\end{equation}

As we shall see in the next proposition, assuming that $(\mathcal{M}^f_t)_{t\in [0,T]}$ is a local martingale for all $f\in C_0^2(\R)$ implies that the stochastic process~$Z$ is a semimartingale.

\begin{proposition}\label{prop:Zsemimart}
  Suppose Assumption~\ref{ass:short}. Let $(X_t)_{t\in [0,T]}$ be an $(\mathcal{F}_t)$-progressively measurable process in $L^1(\Omega\times [0,T])$ and $(Z_t)_{t\in [0,T]}$ be an $(\mathcal{F}_t)$-adapted and continuous process on a filtered probability space $(\Omega,\mathcal{F},(\mathcal{F}_t)_{t\in[0,T]},\P)$ satisfying the usual conditions. If $(\mathcal{M}^f_t)_{t\in[0,T]}$ is a local martingale for every $f\in C_0^2(\R)$, then we have:
  \begin{enumerate}
    \item[(i)] $(Z_t)_{t\in[0,T]}$ is a semimartingale with characteristics $\big(\int_0^{\cdot} \mu(s,X_s)\dd s ,\int_0^{\cdot}\sigma(s,X_s)^2\dd s,0\big)$.
    \item[(ii)] There exists a filtered probability space $(\tilde{\Omega},\tilde{\mathcal{F}},(\tilde{\mathcal{F}}_t)_{t\in[0,T]},\tilde{\P})$ satisfying the usual conditions such that $(Z_t)_{t\in [0,T]}$ is a semimartingale on $(\tilde{\Omega},\tilde{\mathcal{F}},(\tilde{\mathcal{F}}_t)_{t\in[0,T]},\tilde{\P})$ and
    \begin{equation*}
      Z_t=\int_0^t \mu(s,X_s)\dd s + \int_0^t \sigma(s,X_s)\dd B_s,\quad t\in[0,T],
    \end{equation*}
    holds $\tilde{\P}$-a.s., for some Brownian motion $(B_t)_{t\in[0,T]}$ on $(\tilde{\Omega},\tilde{\mathcal{F}},(\tilde{\mathcal{F}}_t)_{t\in[0,T]},\tilde{\P})$.
  \end{enumerate}
\end{proposition}

\begin{proof}
  (i) By \cite[Theorem~II.2.42]{Jacod2003}, in order to prove the assertion, it is sufficient to show that $(\mathcal{M}_t^f)_{t\in[0,T]}$, defined in~\eqref{def:M_f1}, is a local martingale for every bounded function $f\in C^2(\R)$.

  Let $f\in C^2(\R)$ be bounded and define the hitting times
  \begin{equation*}
    \tau_n:=\inf\limits_{t\in[0,T]}\lbrace \max(|X_t|,|Z_t|)\geq n\rbrace,\quad n\in\N.
  \end{equation*}
  Note that $\tau_n\to T$ a.s. as $n\to\infty$ since $X\in L^1(\Omega\times [0,T])$ and $Z$ is continuous. Since the underlying filtered probability space satisfies the usual conditions, by the D{\'e}but theorem (see \cite[Chapter~I, (4.15) Theorem]{Revuz1999}), the hitting times $(\tau_n)_{n\in\N}$ are stopping times. It remains to show that $(\tau_n)_{n\in \N}$ is a localizing sequence for $(\mathcal{M}_t^f)_{t\in[0,T]}$. To that end, we approximate $f$ by the functions $(f_n)_{n\in\N}\subset C^2_0(\R)$ given by $f_n := \phi_n f$ for some $\phi_n\in C_0^2(\R)$ taking values in $[0,1]$ and being identical to~$1$ on $[-n,n]$. Hence, $(\mathcal{M}_t^{f_n})_{t\in[0,T]}$ is a local martingale for every $n\in\N$ and, thus, the stopped process $(\mathcal{M}_{t\wedge \tau_n}^{f_n})_{t\in[0,T]}$, given by
  \begin{equation*}
    \mathcal{M}_{t\wedge\tau_n}^{f_n}=(f_n)(Z_{t\wedge\tau_n})-\int_0^{t\wedge\tau_n} \mathcal{A}^{f_n}(s,X_s,Z_s)\dd s,\quad t\in[0,T],
  \end{equation*}
  is a martingale as
  \begin{equation*}
    |\mathcal{M}_{t\wedge\tau_n}^{f_n}|\leq \sup_{x\in \R}|f(x)| + C_{\sigma, \mu, n} n^2<\infty,
  \end{equation*}
  for some constant $C_{\sigma, \mu, n}>0$, using the definition of $\tau_n$ and the linear growth condition on $\mu$ and $\sigma$. Since $ \mathcal{M}_{t\wedge\tau_n}^{f_n}= \mathcal{M}_{t\wedge\tau_n}^{f}$ for $t\in [0,T]$, $(\mathcal{M}_{t\wedge\tau_n}^{f})_{t\in [0,T]}$ is a martingale for every $n\in \N$ and, hence, $(\tau_n)_{n\in \N}$ a localizing sequence for $(\mathcal{M}^f_t)_{t\in [0,T]}$.

  (ii) Since the process $(Z_t)_{t\in[0,T]}$ is a semimartingale with absolutely continuous characteristics $\big(\int_0^{\cdot} \mu(s,X_s)\dd s ,\int_0^{\cdot}\sigma^2(s,X_s)\dd s,0\big)$, the assertion follow by \cite[Theorem~2.1.2]{Jacod2012}.
\end{proof}

Keeping these preliminary considerations and the classical martingale problem (see e.g. \cite[Definition~7.1.1]{Kallianpur2014}) in mind, we formulate a local martingale problem associated to the stochastic Volterra equation~\eqref{eq:SVE}.

\begin{definition}\label{def:martproblem}
  A \textup{solution to the Volterra local martingale problem} given $(x_0,\mu,\sigma,K_\mu,K_\sigma)$ is a triple $(X,Z)$, $(\Omega,\mathcal{F},\mathbb{P})$, $(\mathcal{F}_t)_{t\in [0,T]}$ such that
  \begin{enumerate}
    \item[(i)] $(\Omega,\mathcal{F},\mathbb{P})$ is a probability space, $(\mathcal{F}_t)_{t\in [0,T]}$ is a filtration of sub-$\sigma$-algebras of $\mathcal{F}$ satisfying the usual conditions,
    \item[(ii)] $X=(X_t)_{t\in [0,T]}\in L^1(\Omega\times [0,T])$ is an $(\mathcal{F}_t)$-progressively measurable process,
    \item[(iii)] $(Z_t)_{t\in[0,T]}$ is a continuous semimartingale with $Z_0=0$ and decomposition $Z=A+M$ for some process $(A_t)_{t\in[0,T]}$ of bounded variation and some local martingale $(M_t)_{t\in[0,T]}$,
    \item[(iv)]the process $(\mathcal{M}^f_t)_{t\in[0,T]}$, given by
    \begin{equation}\label{def:M_f}
      \mathcal{M}^f_t:= f(Z_t)-\int_0^t \mathcal{A}^f(s,X_s,Z_s)\dd s,\quad t\in [0,T],
    \end{equation}
    is a local martingale for every $f\in C_0^2(\R)$, where $\mathcal{A}^f$ is defined as in \eqref{eq:operator A}, and
    \item[(v)] the following equality holds:
    \begin{equation}\label{eq:X_MP}
      X_t =x_0(t)+\int_0^t K_\mu(s,t)\dd A_s + \int_0^tK_\sigma(s,t)\dd M_s, \quad t\in[0,T],
      \quad\P\text{-a.s}.
    \end{equation}
  \end{enumerate}
\end{definition}

\begin{remark}
  The first Volterra local martingale problem was formulated in \cite{AbiJaber2021} for stochastic Volterra equations of convolution type, that is, the kernels $K_\mu, K_{\sigma}$ are supposed to be of the form $K(t-s)$ for a deterministic function $K\colon [0,T]\to \R$, see \cite[Definition~3.1]{AbiJaber2021}. However, \cite[Definition~3.1]{AbiJaber2021} fundamentally relies on the convolutional structure to ensure that a weak solution to the SVE leads to a solution of the Volterra local martingale problem. The latter conclusion is based on a substitution and stochastic Fubini argument, which is not applicable for general kernels. Compared to \cite[Definition~3.1]{AbiJaber2021}, the essential difference is that we reformulated \cite[(3.3)]{AbiJaber2021} to the condition~\eqref{eq:X_MP}. While both conditions are equivalent for kernels of convolutional type, the advantage of \eqref{eq:X_MP} is that it allows for general kernels.

  Moreover, notice that the Volterra local martingale problem as presented in Definition~\ref{def:martproblem} reduces to the local martingale problem for ordinary stochastic differential equations in the case $K_\mu= K_\sigma = 1$. Indeed, in this case conditions (i) and (iv) imply conditions (iii) and (v) on a possibly extended probability space, see Proposition~\ref{prop:Zsemimart}.
\end{remark}

\begin{remark}
  Condition (iii) of Definition~\ref{def:martproblem} can be relaxed to the condition ``$(Z_t)_{t\in [0,T]}$ is an $(\mathcal{F}_t)$-adapted and continuous process'' since this together with (iv) of Definition~\ref{def:martproblem} already implies the semimartingale property of $(Z_t)_{t\in [0,T]}$, see Proposition~\ref{prop:Zsemimart}. However, we decided to directly postulate the semimartingale property of $(Z_t)_{t\in [0,T]}$ in the formulation of the Volterra local martingale problem to ensure that condition~(v) is obviously well-defined.
\end{remark}

As for ordinary stochastic differential equations, the existence of weak solutions to SVEs is equivalent to the solvability of the associated Volterra local martingale problem, like in the case of convolutional SVEs as shown in \cite[Lemma~3.3]{AbiJaber2021}.

\begin{lemma}\label{lem:equivalence}
  Suppose Assumption~\ref{ass:short}. There exists a weak solution to the SVE~\eqref{eq:SVE} if and only if there exists a solution to the Volterra local martingale problem given $(x_0,\mu,\sigma,K_\mu,K_\sigma)$.
\end{lemma}

\begin{proof}
  Let $(X,B)$ be a (weak) solution to~\eqref{eq:SVE} on a probability space $(\Omega,\mathcal{F},\P)$. Setting
  \begin{equation*}
    Z_t:=A_t+M_t:=\int_0^t \mu(s,X_s)\dd s+\int_0^t \sigma(s,X_s)\dd B_s,\quad t\in [0,T],
  \end{equation*}
  It{\^o}'s formula applied to $f(Z_t)$ for $f\in C_0^2(\R)$ yields that
  \begin{align*}
    \mathcal{M}_t^f
    &=f(Z_t)-\int_0^t f^\prime(Z_s)\mu(s,X_s)\dd s-\frac{1}{2}\int_0^t f^{\prime\prime}(Z_s)\sigma(s,X_s)^2\dd s\\
    &=f(Z_0)+\int_0^t f^\prime(Z_s)\sigma(s,X_s)\dd B_s,
  \end{align*}
  which is a local martingale and, by its definition, $Z$ is a semimartingale satisfying~\eqref{eq:X_MP}.

  Conversely, if there exists a solution to the Volterra local martingale problem, we obtain a weak solution to the SVE~\eqref{eq:SVE} by using \eqref{eq:X_MP} and Proposition~\ref{prop:Zsemimart}, which yields that $A_t=\int_0^t \mu(s,X_s)\dd s$ and $M_t=\int_0^t \sigma(s,X_s)\dd B_s$ for some Brownian motion~$(B_t)_{t\in [0,T]}$.
\end{proof}

\section{Existence of weak solutions}\label{sec:construction}

In this section we establish the existence of a weak solution to the SVE~\eqref{eq:SVE} and, equivalently, of a solution to the associated Volterra local martingale problem, under suitable assumptions on the initial condition, coefficients and kernels, which we state in the following.

\begin{assumption}\label{ass:all}
  There is some $p\in (4,\infty)$ and some $\gamma\in (\frac{2}{p},\frac{1}{2})$ such that:
  \begin{enumerate}
    \item[(i)] There is a constant $C_p>0$ such that, for all $(t,t^\prime)\in \Delta_T$,
    \begin{align}\label{eq:regularity kernels}
    \begin{split}
      &\int_0^t |K_{\mu}(s,t')-K_{\mu}(s,t)|^{\frac{p}{p-1}}\dd s + \int_t^{t'} |K_{\mu}(s,t')|^{\frac{p}{p-1}}\dd s \leq C_p|t'-t|^{\frac{\gamma p}{p-1}},\\
      &\int_0^t |K_{\sigma}(s,t')-K_{\sigma}(s,t)|^{\frac{2p}{p-2}}\dd s  + \int_t^{t'} |K_{\sigma}(s,t')|^{\frac{2p}{p-2}}\dd s  \leq C_p|t'-t|^{\frac{2\gamma p}{p-2}}.
    \end{split}
    \end{align}
    \item[(ii)] The coefficients $\mu,\sigma\colon [0,T]\times\R\to\R$ are measurable functions such that for every compact set $\mathcal{K}\subset \R$ and every $\epsilon>0$ there exists a $\delta>0$ such that
    \begin{equation*}
      |\mu(t,x)-\mu(t,y)|+|\sigma(t,x)-\sigma(t,y)|\leq \epsilon, \quad t\in[0,T],\,
      x,y\in \mathcal{K} \text{ with } |x-y|\leq\delta,
    \end{equation*}
    and $\mu,\sigma$ fulfill the linear growth condition
    \begin{equation}\label{eq:linear growth}
      |\mu(t,x)|+|\sigma(t,x)| \leq C_{\mu,\sigma} (1 +|x|),\quad t\in [0,T],\, x\in \R,
    \end{equation}
    for a constant $C_{\mu,\sigma}>0$
    \item[(iii)] The initial condition $x_0\colon [0,T]\to\R$ is $\beta$-H{\"o}lder continuous for every $\beta\in (0,\gamma-1/p)$.
  \end{enumerate}
\end{assumption}

Note that Assumption~\ref{ass:all} directly implies \cite[(3.1)]{Promel2022} with the choice $\epsilon=\frac{2p}{p-2}-2$, and vice versa \cite[(3.1)]{Promel2022} implies Assumption~\ref{ass:all}~(i) with $p=4/\epsilon+2$ (and if necessary rescaling the exponent using H{\"o}lder's inequality if $\epsilon\geq 2$ to secure $p>4$).

\medskip

To formulate our second assumption, for a measurable function $K\colon\Delta_T\to\R$, we say $K(\cdot,t)$ is absolutely continuous for every $t\in [0,T]$ if there exists an integrable function $\partial_1 K\colon \Delta_T \to\R$ such that $K(s,t)-K(0,t)=\int_0^s \partial_1 K(u,t)\dd u$ for $(s,t)\in\Delta_T$.

\begin{assumption}\label{ass:kernel_sm_or_conv}
  The kernel~$K_\mu$ is measurable and bounded in $L^1([0,T])$ uniformly in the second variable, i.e.
  \begin{equation*}
    \sup_{t\in[0,T]}\int_0^t|K_\mu(s,t)|\dd s \leq C
  \end{equation*}
  for some constant $C>0$.
  The kernel $K_\sigma$ is measurable and satisfies at least one of the following conditions:
  \begin{itemize}
    \item[(i)] $K_\sigma$ is a bounded function and $K_\sigma(\cdot,t)$ is absolutely continuous for every $t\in[0,T]$ such that $\partial_1 K_\sigma$ fulfills
    \begin{equation*}
      \sup\limits_{t\in[0,T]}\bigg| \int_0^{t} |\partial_1K_\sigma(s,t)|^p\dd s \bigg|^{\frac{1}{p}}\leq C
    \end{equation*}
    for some $p>1$ and some constant $C>0$.
    \item[(ii)] $K_\sigma(s,t)=\tilde{K}(t-s)$ for all $(s,t)\in\Delta_T$ for a function $\tilde{K}\in L^2([0,T])$.
  \end{itemize}
\end{assumption}

Note, that Assumption~\ref{ass:kernel_sm_or_conv} is satisfied by every convolutional kernel $K_\mu(s,t)=\tilde{K}(t-s)$ for all $(s,t)\in\Delta_T$ for a function $\tilde{K}\in L^1([0,T])$, and in case of Assumption~\ref{ass:kernel_sm_or_conv}~(i), the bound on the second summand in \eqref{eq:regularity kernels} is trivially fulfilled. With these assumptions at hand we are ready to state our main result.

\begin{theorem}\label{thm: weak existence}
  Suppose Assumptions~\ref{ass:all} and~\ref{ass:kernel_sm_or_conv}. Then, there exists a weak solution (in the sense of Definition~\ref{def:weak solution}) such that $(X_t)_{t\in[0,T]}\in C([0,T];\R)$ to the stochastic Volterra equation~\eqref{eq:SVE}.
\end{theorem}

Before proving the aforementioned existence result, let us briefly discuss some properties of weak solutions to the SVE~\eqref{eq:SVE} and some exemplary kernels.

\begin{remark}\label{rem:regularity}
  Suppose Assumption~\ref{ass:all}. Due to \cite[Lemma~3.4 and~Corollary~3.5]{Promel2022}, any weak solution such that $(X_t)_{t\in [0,T]}\in C([0,T];\R)$ to the SVE~\eqref{eq:SVE} satisfies $\sup_{t\in[0,T]}\E[|X_t|^q]<\infty$ for any $q\in [1,\infty)$ and possesses a $\beta$-H{\"o}lder continuous modification for any $\beta\in (0,\gamma-1/p)$.
\end{remark}

\begin{remark}
  Assumptions~\ref{ass:all} and~\ref{ass:kernel_sm_or_conv} are satisfied, e.g., by the following type of diffusion kernels:
  \begin{enumerate}
    \item[(i)] $K_\sigma(s,t):=(t-s)^{-\alpha}$ for $\alpha\in(0,\frac{1}{2})$ for any $p\in (\frac{6}{1-2\alpha},\infty)$ with $\gamma=\frac{1}{2}-\alpha-\frac{1}{p}$,
    \item[(ii)] $K_\sigma(s,t):=\tilde{K}(t-s)$ for a Lipschitz continuous function $\tilde{K}\colon [0,T]\to \R$,
    \item[(iii)] kernels fulfilling \cite[Assumption~2.1]{Promel2022}, and
    \item[(iv)] weakly differentiable kernels such that $\partial_1 K_\sigma(s,t)\leq C(t-s)^{-\alpha}$ for $\alpha\in(0,\frac{1}{2})$.
  \end{enumerate}
\end{remark}

The remainder of the paper is devoted to implement the proof of Theorem~\ref{thm: weak existence} based on several auxiliary lemmas. Generally speaking, the presented proof follows the classical approach of approximation the coefficients by Lipschitz continuous coefficients, in combination with a tightness argument. In contrast, \cite{AbiJaber2019} uses an approximation of the driving noise by pure jump processes with finite activity, which allows to treat convolutional SVEs with jumps.

Note that Lemma~\ref{lem:solution to Volterra martingale problem} implies Theorem~\ref{thm: weak existence} due to Lemma~\ref{lem:equivalence}. Note further that the continuity of $(X_t)_{t\in[0,T]}$ in Theorem~\ref{thm: weak existence} follows by the convergence $\hat{X}^k\to X$ in $C([0,T];\R)$ in Lemma~\ref{lem:limitprocesses}.

\medskip

Assuming the coefficients $\mu,\sigma$ satisfy Assumption~\ref{ass:all}, the next lemma provides a way to approximate $\mu,\sigma$ locally uniformly by Lipschitz continuous coefficients.

\begin{lemma}\label{lem:approximation}
  Let $f\colon [0,T]\times \R \to \R$ be a measurable function such that for every compact set $\mathcal{K}\subset \R$ and every $\epsilon>0$ there exists $\delta>0$ such that
  \begin{equation*}
    |f(t,x)-f(t,y)|\leq \epsilon, \quad t\in[0,T],\,x,y\in \mathcal{K} \text{ with } |x-y|\leq \delta,
  \end{equation*}
  and such that $f$ fulfills the linear growth condition
  \begin{equation}\label{ineq:glgrowth}
    |f(t,x)|\leq C_f(1+|x|), \quad t\in [0,T],\, x\in\R,
  \end{equation}
  for some constant $C_f>0$. Then, there is a sequence $(f_n)_{n\in \N}$ of measurable functions $f_n\colon [0,T]\times \R \to \R$, which satisfies:
  \begin{itemize}
    \item[(i)] linear growth: for $C_f>0$ as in \eqref{ineq:glgrowth}, we have
    \begin{equation*}
      |f_n(t,x)|\leq 2C_f (1+|x|), \quad t\in [0,T],\, x\in \R;
    \end{equation*}
    \item[(ii)] Lipschitz continuity: for each $n\in \N$ there is a $C_n>0$ such that
    \begin{equation*}
      |f_n(t,x)-f_n(t,y)|\leq C_n|x-y|, \quad t\in [0,T],\, x,y\in \R;
    \end{equation*}
    \item[(iii)] locally uniform convergence: for all $r\in (0,\infty)$ we have
    \begin{equation*}
      \sup\limits_{t\in[0,T],x\in [-r,r]}|f(t,x)-f_n(t,x)|\to 0,\quad \text{as }n\to\infty.
    \end{equation*}
  \end{itemize}
\end{lemma}

\begin{proof}
  We explicitly choose the sequence $(f_n)_{n\in\N}$ by
  \begin{equation*}
    f_n(t,x):= \phi_n(x) \int_{\R} f(t,x-y)\delta_n(y)\dd y,\quad n\in\N,
  \end{equation*}
  for some $\phi_n\in C_0^2(\R)$ with support in $[-(n+1),n+1]$, taking values in $[0,1]$ and being identical to~$1$ on $[-n,n]$, where $\delta_n(y):=\frac{1}{c_n}(1-y^2)^n \mathbf{1}_{[-1,1]}(y)$ with $c_n:=\int_{[-1,1]} (1-y^2)^n\dd y$.

  (i) For $t\in[0,T]$ and $x\in \R$, using the linear growth condition on $f$, we get
  \begin{align*}
    |f_n(t,x)|
    \leq C_f\int_{[-1,1]}(1+|x-y|)\delta_n(y)\dd y
    \leq  C_f\int_{[-1,1]}(2+|x|)\delta_n(y)\dd y\leq 2 C_f(1+|x|).
  \end{align*}

  (ii) Let $t\in[0,T]$, $x,y\in\R$ and $n\in\N$. Using the compact support of $f_n$ and the fact, that every $\delta_n$ is Lipschitz continuous as a smooth function with compact support, we get
  \begin{equation*}
    \big| f_n(t,x)-f_n(t,y) \big|
    \leq C_f c_n |x-y|\int_{-(n+2)}^{n+2}(1+|z|)\dd z\leq C_n |x-y|
  \end{equation*}
  for some constant~$C_n$.

  (iii) Due to the continuity property of~$f$, we can find for every $r>0$ and for every $\epsilon>0$ some $\delta >0$ such that for all $x,y\in [-r,r]$ with $|x-y|\leq\delta$ and all $t\in[0,T]$ holds $|f(t,x)-f(t,y)|\leq \epsilon$. Assuming $n\in\N$ to be large enough that $\phi_n\equiv 1$ on $[-r,r]$, we get for any $x\in [-r,r]$,
  \begin{align*}
    &|f(t,x)-f_n(t,x)| \\
    &\quad = \int_{[-\delta,\delta]}\delta_n(y)\big|f(t,x)-f(t,x-y)\big|\dd y + \int_{[-1,1]\setminus [-\delta,\delta]}\delta_n(y)\big|f(t,x)-f(t,x-y)\big|\dd y.
  \end{align*}
  Let now $N(\epsilon,r)>0$ be big enough, such that $\int_{[-1,1]\setminus [-\delta,\delta]}\delta_n(y)\dd y<\epsilon$ and $\phi_n\equiv 1$ on $[-r,r]$ for all $n\geq N(\epsilon,r)$. Then, setting $\tilde{r} :=r+1$ for all $n\geq N(\epsilon,r)$
  \begin{align*}
    |f(t,x)-f_n(t,x)|&\leq \int_{[-\delta,\delta]}\delta_n(y)\epsilon\dd y + 2\epsilon\sup_{\substack{s\in[0,T],\\ \tilde{x}\in [-\tilde{r} ,\tilde{r} ]}}|f(s,\tilde{x})|
    \leq \epsilon\bigg( 1+ 2\sup_{\substack{s\in[0,T],\\ \tilde{x}\in [-\tilde{r} ,\tilde{r} ]}}|f(s,\tilde{x})|\bigg),
  \end{align*}
  which tends to zero as $\epsilon\to 0$.
\end{proof}

A suitable approximation, like the one provided in Lemma~\ref{lem:approximation}, ensures the convergence of associated Riemann--Stieltjes integrals. We denote by $C([0,T];\R)$ the space of all continuous functions $g\colon [0,T]\to \R$, which is equipped with the supremum norm~$\|\cdot \|_{\infty}$.

\begin{lemma}\label{lem:convergence_integral}
  Let $f\colon [0,T]\times\R\to\R$ be a function such that for every compact set $\mathcal{K}\subset \R$ and every $\epsilon>0$ there exists $ \delta>0$ such that
  \begin{equation}\label{cont_f}
    |f(t,x)-f(t,y)|\leq \epsilon, \quad t\in[0,T],\,
    x,y\in \mathcal{K} \text{ with } |x-y|\leq\delta,
  \end{equation}
  and $(f_k)_{k\in\N}$ be a sequence of functions such that $f_k\colon[0,T]\times\R\to\R$ and $|f(t,x)|+|f_k(t,x)|\leq C(1+|x|^2)$, $t\in[0,T]$, $x\in\R$, for all $k\in\N$ and for some $C>0$, and $f_k\to f$ locally uniformly. Let $K\colon \Delta_T \to \R$ be measurable and bounded in $L^1([0,T])$ uniformly in the second variable, i.e. $ \sup_{t\in[0,T]}\int_0^t|K(s,t)|\dd s \leq M$ for some $M>0$. If $(X^k)_{k\in\N}$ is a sequence of continuous stochastic processes such that $X^k\to X$ in $C([0,T];\R)$ as $k\to \infty$ $\P$-a.s, then
  \begin{align*}
    &\Big( \int_0^{\cdot}K(s,\cdot) f_k(s,X^k_s)\dd s \Big)_{t\in[0,T]} \stackrel{\P}{\to} \Big( \int_0^{\cdot}K(s,\cdot) f(s,X_s)\dd s \Big)_{t\in[0,T]}
    \quad \text{w.r.t. }  \|\cdot\|_{\infty}, \quad k\to \infty,
  \end{align*}
  where $\stackrel{\P}{\to}$ denotes convergence in probability.
\end{lemma}

\begin{proof}
  First, note that due to the continuity condition \eqref{cont_f}, for every $n\in\N$ there exists some continuous non-decreasing function $g_n\colon [0,\infty)\to[0,\infty)$ with $g_n(0)=0$, such that for all $x,y\in[-n,n]$,
  \begin{equation*}
    |f(t,x)-f(t,y)|\leq g_n(|x-y|),\quad t\in[0,T].
  \end{equation*}
  Let $\epsilon>0$ and $\delta>0$ be fixed but arbitrary. Choose $N\in\N$ and $K\in\N$ big enough such that
  \begin{equation*}
    \P\big( \|X\|_{\infty}\geq n/2 \big)\leq\delta/4\qquad \text{and}\qquad \P\big( \|X^k-X\|_{\infty}\geq n/2 \big)\leq\delta/4,
  \end{equation*}
  for all $n\geq N$ and $k\geq K$. Then,
  \begin{align*}
    \P\big( \|X\|_{\infty}\vee \|X^k\|_{\infty}\geq n \big)&\leq \P\Big( \lbrace \|X\|_{\infty}\geq n\rbrace \cup\lbrace \|X^k-X\|_{\infty}+\|X\|_{\infty}\geq n\rbrace \Big)\\
    &\leq \P\big( \|X\|_{\infty}\geq n/2\big) +\P\big( \|X^k-X\|_{\infty}\geq n/2 \big)\\
    &\leq\delta/4+\delta/4=\delta/2.
  \end{align*}
  For every $n,k\in\N$, on $\lbrace \|X\|_{\infty}\vee \|X^k\|_{\infty} \leq n\rbrace$ we can bound for $t\in[0,T]$,
  \begin{align}\label{eq:integrateK}
    A^k_t-A_t&:=\int_0^t K(s,t) f_k(s,X^k_s)\dd s-\int_0^t  K(s,t) f(s,X_s) \dd s\notag\\
    &\leq \int_0^t |K(s,t)|\big|f_k(s,X^k_s)-f(s,X^k_s)\big|\dd s + \int_0^t |K(s,t)| \big| f(s,X^k_s)-f(s,X_s) \big|\dd s\notag\\
    &\leq M\bigg( \sup_{t\in [0,T],\, x\in [-n,n]} |f_k(t,x)-f(t,x)| + g_{n}\big(\|X^k-X\|_{\infty}\big) \bigg),
  \end{align}
  with $\sup_{t\in[0,T]}\int_0^t|K(s,t)|\dd s \leq M$.
  For every $n\in\N$ we choose $K_{\epsilon\delta}^n\in\N$ sufficiently large such that
  \begin{equation*}
    \P\bigg( \sup_{t\in [0,T],\, x\in [-n,n]} |f_k(t,x)-f(t,x)|+g_{ n}(\|X^k-X\|_{\infty})\geq \epsilon/M \bigg)\leq\delta/2,\quad  k\geq K_{\epsilon\delta}^n.
  \end{equation*}
  Setting $K_{\epsilon\delta}:=\max\lbrace K_{\epsilon\delta}^N,K \rbrace$, we get
  \begin{align*}
    &\P\big( \|A^k-A\|_{\infty}\geq \epsilon \big)\\
    &\quad\leq \P\Big(\lbrace \|A^k-A\|_{\infty}\geq \epsilon \rbrace\cap \lbrace \|X\|_{\infty}\vee \|X^k\|_{\infty}<N \rbrace \Big) + \P\big( \|X\|_{\infty}\vee \|X^k\|_{\infty}\geq N \big)\\
    &\quad\leq \P\Big( \sup_{t\in [0,T],\, x\in [-N,N]} |f_k(t,x)-f(t,x)| +g_{N}(\|X^k-X\|_{\infty})\geq \epsilon/M \Big) + \delta/2 \leq \delta,
  \end{align*}
  for all $k\geq K_{\epsilon\delta}$, which shows the desired convergence.
\end{proof}

Given coefficients $\mu, \sigma$ satisfying Assumption~\ref{ass:all}, we fix, relying on Lemma~\ref{lem:approximation}, two sequences $(\mu_n)_{n\in \N}$ and $(\sigma_n)_{n\in \N}$ with
\begin{equation*}
  \mu_n \colon [0,T]\times \R \to \R
  \quad \text{and}\quad
  \sigma_n \colon [0,T]\times \R \to \R,
\end{equation*}
that fulfill properties (i)-(iii) of Lemma~\ref{lem:approximation}. For every $n\in \N$, we define $(X^n_t)_{t\in[0,T]}$ as the unique (strong) solution (see e.g. the text before \cite[Theorem~2.3]{Promel2022} for the definition of unique strong solutions to SVEs) to the stochastic Volterra equation
\begin{equation}
  X^n_t=x_0(t) + \int_0^t K_\mu(s,t)\mu_n(s,X^n_s)\dd s + \int_0^t K_\sigma(s,t)\sigma_n(s,X^n_s)\dd B_s, \quad t\in [0,T],\label{def_Xn}
\end{equation}
given a Brownian motion $(B_t)_{t\in[0,T]}$ on some probability space $(\Omega,\mathcal{F},\P)$. Note that $(X^n_t)_{t\in[0,T]}$ exists by \cite[Theorem~1.1]{Wang2008} due to the Lipschitz continuity of $\mu_n$ and $\sigma_n$. Furthermore, we introduce the sequences $(A^n)_{n\in\N}$ and $(M^n)_{n\in\N}$ by
\begin{equation}\label{def:sequencesMnZn}
  A_t^n:=\int_0^t \mu_n(s,X^n_s)\dd s\quad \text{and}\quad M_t^n:=\int_0^t \sigma_n(s,X^n_s)\dd B_s, \qquad t\in[0,T].
\end{equation}
In the following, we denote $X\stackrel{\mathscr{D}}{\sim}Y$ for equality in law of stochastic processes $X$ and $Y$.

\begin{lemma}\label{lem:limitprocesses}
  Suppose Assumption~\ref{ass:all} and let $(X^n)_{n\in\N}$, $(A^n)_{n\in\N}$ and $(M^n)_{n\in\N}$ be given by \eqref{def_Xn} and \eqref{def:sequencesMnZn}. Then, there exist continuous stochastic processes $(\hat{X}^k)_{k\in\N}$, $(\hat{A}^k)_{k\in\N}$, $(\hat{M}^k)_{k\in\N}$, $X$, $A$, $M$ and a Brownian motion $(\tilde{B}_t)_{t\in[0,T]}$ on a common probability space $(\tilde{\Omega},\tilde{\mathcal{F}},\tilde{\P})$ such that $(\hat{X}^k,\hat{A}^k,\hat{M}^k)\to (X,A,M)$ in $C([0,T];\R^3)$ as $k\to \infty$ $\tilde{\P}$-a.s., $(\hat{X}^k,\hat{A}^k,\hat{M}^k)\stackrel{\mathscr{D}}{\sim}(X^{n_k},A^{n_k},M^{n_k})$ and $M$ is a local martingale with the representation
  \begin{equation*}
    M_t=\int_0^t \sigma(s,X_s)\dd \tilde{B}_s,\quad t\in[0,T],
  \end{equation*}
  where $(X^{n_k},A^{n_k},M^{n_k})_{k \in \N}$ denotes some subsequence of $(X^{n},A^{n},M^{n})_{n \in \N}$.
\end{lemma}

\begin{proof}
  First we want to apply Kolmogorov’s tightness criterion (see \cite[Problem~2.4.11]{Karatzas1991}) to the probability measures $(\P_{(X^n,A^n,M^n,B)})_{n\in\N}$ associated to the four-dimensional stochastic processes $(X^n,A^n,M^n,B)_{n\in\N}$. By Lemma~\ref{lem:approximation}~(i) we know, that the coefficients $\mu_n$ and $\sigma_n$ fulfill the linear growth condition \eqref{eq:linear growth} with uniformly bounded constants, i.e. $C_{\mu_n,\sigma_n}\leq 2C_{\mu,\sigma}$ for all $n\in\N$. Hence, using $p\in(4,\infty)$ from Assumption~\ref{ass:all}, we deduce, by \cite[Lemma~3.4]{Promel2022}, that
  \begin{equation*}
    \sup_{n\in \N} \sup_{ s\in [0,T]} \mathbb{E}[|X_s^n|^p]\leq C \bigg(1+\sup\limits_{ s\in[0,T]}|x_0(s)|\bigg)^p<\infty,
  \end{equation*}
  and, by \cite[Lemma~3.1 and Remark~3.3]{Promel2022}, that
  \begin{equation*}
    \mathbb{E}[|X_{t'}^n-x_0(t')-X_t^n-x_0(t)|^p] \leq C|t'-t|^{\beta p}, \quad n\in \N,
  \end{equation*}
  for every $\beta \in (0, \gamma-1/p)$, where the constant $C>0$ depends only on $p$, $T$, $K_{\mu}$, $K_{\sigma}$ and $C_{\mu,\sigma}$. Moreover, it is straightforward to show that
  \begin{equation*}
    \E[|A_{t'}^n-A_t^n|^p]\leq C|t'-t|^{\frac{p}{2}}\quad\text{and}\quad\E[|M_{t'}^n-M_t^n|^p]\leq C|t'-t|^{\frac{p}{2}}
  \end{equation*}
  for all $0\leq t\leq t'\leq T$ and some constant $C>0$, by H{\"o}lder's inequality and Burkholder--Davis--Gundy's inequality, respectively. Choosing $\beta$ sufficiently close to $\gamma-1/p$ so that $\beta p>1$, which is possible due to $\gamma>2/p$ in Assumption~\ref{ass:all}, and noting that the initial distributions $(X^n_0,A^n_0,M^n_0,B_0)_{n\in \N}$ are independent of $n$, we can apply Kolmogorov's tightness criterion to obtain the tightness of the sequence $(\P_{(X^n,A^n,M^n,B)})_{n\in\N}$. Hence, by Prohorov's theorem (\cite[Theorem~2.4.7]{Karatzas1991}) we get relative compactness (\cite[Definition~2.4.6]{Karatzas1991}) of the sequence of measures $(\P_{(X^n,A^n,M^n,B)})_{n\in\N}$ in $\mathcal{M}_1(C([0,T];\R^4))$, which denotes the space of all probability measures on $C([0,T];\R^4)$. Consequently, there exists a converging subsequence $(\P_{(X^{n_k},A^{n_k},M^{n_k},B)})_{k\in\N}$ such that
  \begin{equation*}
    \P_{(X^{n_k},A^{n_k},M^{n_k},B)} \to \P_{(X,A,M,B)}\quad \text{weakly} \quad \text{as}\quad k\to \infty,
  \end{equation*}
  for some measure  $\P_{(X,A,M,B)}$ in $\mathcal{M}_1(C([0,T];\R^4))$.

  The Skorokhod representation theorem (see e.g. \cite[Theorem~11.7.2]{Dudley2002}) yields the existence of some probability space $(\hat{\Omega},\hat{\mathcal{F}},\hat{\P})$ with continuous stochastic processes $(\hat{X}^k)_{k\in\N}$, $(\hat{A}^k)_{k\in\N}$, $(\hat{M}^k)_{k\in\N}$, $(\hat{B}^k)_{k\in\N}$ and $X$, $A$, $M$, $\hat{B}$ on it such that
  \begin{equation*}
    (X^{n_k},A^{n_k},M^{n_k},\hat{B})\stackrel{\mathscr{D}}{\sim}(\hat{X}^k,\hat{A}^k,\hat{M}^k,\hat{B}^k),\qquad k\in\N,
  \end{equation*}
  and
  \begin{equation*}
    (\hat{X}^k,\hat{A}^k,\hat{M}^k,\hat{B}^k)\to (X,A,M,\hat{B}) \quad \text{in}\quad C([0,T];\R^4) \quad  \text{as}\quad k\to\infty, \quad \hat{\P}\text{-a.s.}
  \end{equation*}
  From a general version of the Yamada--Watanabe result, see \cite[Theorem~1.5]{Kurtz2014}, we can deduce that $\hat{M}^k_t=\int_0^t\sigma_{n_k}(s,\hat{X}^k_s)\dd \hat{B}^k_s$, for $t\in[0,T]$ and for all $k\in\N$, and the stochastic processes $(B^k)_{k\in\N}$ are Brownian motions as $\hat{B}^k\stackrel{\mathscr{D}}{\sim}B$. Thus, $\hat{M}^k$ is a local $\hat{\P}$-martingale with quadratic variation $\langle\hat{M}^k\rangle_t=\int_0^t\sigma_{n_k}(s,\hat{X}^k_s)^2\dd s$.

  Due to the $\hat{\P}$-a.s. convergence of $(\hat{M}^k)_{k\in \N}$ to $M$, \cite[Proposition~IX.1.17]{Jacod2003} implies that $M$ is also a local $\hat\P$-martingale, and the convergence of $\int_0^t \sigma_{n_k}(s,X^{n_k}_s)^2\dd s$ in probability, see Lemma~\ref{lem:convergence_integral}, together with \cite[Corollary~VI.6.29]{Jacod2003} implies that the quadratic variation of $M$ is $\langle M\rangle_t=\int_0^t\sigma(s,X_s)^2\dd s$. Therefore, the representation theorem for local martingales with absolutely continuous quadratic variations (see e.g. \cite[Theorem~3.4.2]{Karatzas1991}) yields the existence of some probability space $(\tilde{\Omega},\tilde{\mathcal{F}},\tilde{\P})$, which is an extension of $(\hat{\Omega},\hat{\mathcal{F}},\hat{P})$, and a Brownian motion $(\tilde{B}_t)_{t\in[0,T]}$ on it, such that $M_t=\int_0^t \sigma(s,X_s)\dd \tilde{B}_s$ for $t\in[0,T]$.
\end{proof}

Using the stochastic processes $X$, $A$ and $M$ from Lemma~\ref{lem:limitprocesses}, we can construct a solution to the Volterra local martingale problem in the sense of Definition~\ref{def:martproblem}.

\begin{lemma}\label{lem:solution to Volterra martingale problem}
  Suppose Assumptions~\ref{ass:all} and \ref{ass:kernel_sm_or_conv}. There exists a solution to the Volterra local martingale problem given $(x_0,\mu,\sigma,K_\mu,K_\sigma)$.
\end{lemma}

\begin{proof}
  Recall, the stochastic processes $(X^n)_{n\in\N}$, $(A^n)_{n\in\N}$ and $(M^n)_{n\in\N}$ on $(\Omega,\mathcal{F},\P)$ are given in \eqref{def:sequencesMnZn} and $(\hat{X}^k)_{k\in\N}$, $(\hat{A}^k)_{k\in\N}$, $(\hat{M}^k)_{k\in\N}$, $X$, $A$ and $M$ on $(\tilde{\Omega},\tilde{\mathcal{F}},\tilde{\P})$ are given by Lemma~\ref{lem:limitprocesses}. We introduce the stochastic processes $(Z^n)_{n\in\N}$, $(\hat{Z}^k)_{k\in\N}$ and $Z$ by
  \begin{equation*}
    Z_t^n := A_t^n + M_t^n,\quad \hat{Z}_t^k:=\hat{A}_t^k + \hat{M}_t^k \quad\text{and}\quad Z_t:= A_t + M_t, \quad t\in[0,T].
  \end{equation*}
  We shall show that the triple $(X,Z)$, $(\tilde{\Omega},\tilde{\mathcal{F}},\tilde{\P})$, $(\mathcal{F}^X_t)_{t\in [0,T]}$, where $(\mathcal{F}^X_t)_{t\in [0,T]}$ denotes the augmented natural filtration of $X$ (cf. \cite[Definition~2.7.2]{Karatzas1991}), solves the Volterra local martingale problem given $(x_0,\mu,\sigma,K_\mu,K_\sigma)$. Since the properties (i)-(iii) of Definition~\ref{def:martproblem} are fairly easy to check, we verify here that
  \begin{itemize}
    \item[(iv)] the process $(\mathcal{M}_t^f)_{t\in[0,T]}$ defined by \eqref{def:M_f} is a local $\tilde{\P}$-martingale for every $f\in C_0^2(\R)$,
    \item[(v)] the equality \eqref{eq:X_MP} holds $\tilde{\P}$-a.s.
  \end{itemize}

  \smallskip

  (iv) For $k\in\N$ and $f\in C_0^2(\R)$, the stochastic process $(\mathcal{M}^{f,k}_t)_{t\in[0,T]}$ is defined by
  \begin{equation*}
    \mathcal{M}^{f,k}_t := f(\hat{Z}_t^k)-\int_0^t \mathcal{A}^{f,k}(s,\hat{X}_s^k,\hat{Z}_s^k)\dd s,\quad t\in [0,T],
  \end{equation*}
  where $\mathcal{A}^{f,k}(t,x,z):=\mu_{n_k}(t,x)f^\prime(z)+\frac{1}{2}\sigma_{n_k}(t,x)^2f^{\prime\prime}(z)$. Due to $(\hat{X}^k,\hat{Z}^k)\stackrel{\mathscr{D}}{\sim}(X^{n_k},Z^{n_k})$ and since $(X^{n_k},Z^{n_k})$ solves the Volterra local martingale problem given $(x_0,\mu_{n_k},\sigma_{n_k},K_\mu,K_\sigma)$ on $(\Omega,\mathcal{F},\P)$ by construction and Lemma~\ref{lem:equivalence}, it follows that $(\mathcal{M}^{f,k}_t)_{t\in[0,T]}$ is a local martingale on $(\tilde{\Omega},\tilde{\mathcal{F}},\tilde{\P)}$ for every $k\in\N$. Moreover, Lemma~\ref{lem:convergence_integral} implies that $\mathcal{M}^{f,k}\to \mathcal{M}^f$ weakly as $k\to \infty$ and, thus, by \cite[Proposition~IX.1.17]{Jacod2003}, the limiting process $(\mathcal{M}^{f}_t)_{t\in[0,T]}$ is a local martingale on $(\tilde{\Omega},\tilde{\mathcal{F}},\tilde{\P})$.

  \smallskip

  (v) Since $(\hat{X}^k,\hat{M}^k)\stackrel{\mathscr{D}}{\sim}(X^{n_k},M^{n_k})$ for every $k\in\N$ and pathwise uniqueness holds for SVEs with Lipschitz continuous coefficients (see e.g. \cite[Theorem~1.1]{Wang2008}), the general version of the Yamada--Watanabe result (\cite[Theorem~1.5]{Kurtz2014}) yields that $\hat{X}^k$ can be represented as the stochastic output of the Volterra equation \eqref{eq:SVE} from the stochastic input $\hat{M}^k$ in the same way as $X^{n_k}$ from $M^{n_k}$, hence, we get that
  \begin{equation}\label{eq:sve_skorokhod}
    \hat{X}^k_t
    = x_0(t)+\int_0^tK_\mu(s,t)\mu_{n_k}(s,\hat{X}^k_s)\dd s+\int_0^tK_\sigma(s,t)\dd \hat{M}^k_s,\quad t\in[0,T],\quad \tilde{\P}\text{-a.s.,}
  \end{equation}
  holds. To continue the proof of (v), we need to distinguish between (a) bounded kernels and (b) kernels of convolutional type.

  (a) We start with the bounded kernels as in Assumption~\ref{ass:kernel_sm_or_conv}~(i). Due to the absolute continuity of $K_\sigma$ in the first variable, we can apply the integration by part formula for semimartingales (see \cite[Theorem~(VI).38.3]{Rogers2000}) to rewrite \eqref{eq:sve_skorokhod} to
  \begin{align}\label{eq:takeLimInprob}
    \hat{X}^k_t= x_0(t)+\int_0^tK_\mu(s,t)\mu_{n_k}(s,\hat{X}^k_s)\dd s+K_\sigma(t,t)\hat{M}^k_t
    +\int_0^t\hat{M}^k_s \partial_1 K_\sigma(s,t)\dd s.
  \end{align}
  Since $(\hat{X}^k,\hat{M}^k)\to (X,M)$ in $C([0,T];\R^2)$ as $k\to\infty$, $\tilde{\P}$-a.s., and $K_\sigma$ is bounded, we obtain by Lemma~\ref{lem:convergence_integral} that $\hat{X}^k\to X$ and $K_\sigma \hat{M}^k\to K_\sigma M$ in $C([0,T];\R)$ as $k\to\infty$, $\tilde{\P}$-a.s., and $\int_0^{\cdot} K_\mu(s,\cdot)\mu_{n_k}(s,\hat{X}^k_s)\dd s \to \int_0^t K_\mu(s,\cdot)\dd A_s$ in $C([0,T];\R)$ in probability as $k\to\infty$. Furthermore, applying H{\"o}lder's inequality with $p>4$ (see Assumption~\ref{ass:kernel_sm_or_conv}) and denoting $q=p/(p-1)$, we get by the integrability of $\partial_1K_\sigma$ that
  \begin{align*}
    \bigg\|\int_0^{\cdot}(\hat{M}^k_s-M_s)\partial_1 K_\sigma(s,\cdot)\dd s\bigg\|_{\infty}
    &\leq \bigg(\int_0^T |\hat{M}_s^k-M_s|^q\dd s\bigg)^{\frac{1}{q}}\bigg\| \int_0^{\cdot} |\partial_1K_\sigma(s,\cdot)|^p\dd s \bigg\|^{\frac{1}{p}}_{\infty}\\
    &\leq C\|\hat{M}^k-M\|_{\infty}.
  \end{align*}
  Hence, the $\tilde{\P}$-a.s. convergence $(\hat{M}^k)_{k\in \N}$ to $M$ implies $\int_0^{\cdot}\hat{M}^k_sK_\sigma(s,\cdot)\dd s\to \int_0^{\cdot} M_sK_\sigma(s,\cdot)\dd s$ as $k\to \infty$ $\tilde{\P}$-a.s., and we can take the limit in probability in \eqref{eq:takeLimInprob} or the $\tilde{\P}$-a.s. limit for some subsequence, to obtain that \eqref{eq:X_MP} holds $\tilde{\P}$-a.s.

  (b) For convolution kernels as in Assumption~\ref{ass:kernel_sm_or_conv}~(ii), we integrate both sides of \eqref{eq:sve_skorokhod} and use the stochastic Fubini theorem (see e.g. \cite[Theorem~2.2]{Veraar2012}) twice to obtain
  \begin{align}
    \int_0^t \hat{X}^k_s\dd s
    &= \int_0^t x_0(s)\dd s +\int _{0}^{t} \int _{0}^{s} K_{\mu}(s,u)\,\mathrm{d} \hat{A}_{u}^{k} \,\mathrm{d}s +\int_0^t \int_0^s K_\sigma(s-u)\dd \hat{M}^k_u \dd s\notag\\
    &= \int_0^t x_0(s)\dd s +\int _{0}^{t} \int _{0}^{s} K_{\mu}(s,u)\,\mathrm{d} \hat{A}_{u}^{k} \,\mathrm{d}s  +\int_0^t \int_u^t K_\sigma(s-u)\dd s\dd \hat{M}^k_u \notag\\
    &= \int_0^t x_0(s)\dd s +\int _{0}^{t} \int _{0}^{s} K_{\mu}(s,u)\,\mathrm{d} \hat{A}_{u}^{k} \,\mathrm{d}s +\int_0^t \int_0^{t-u} K_\sigma(s)\dd s\dd \hat{M}^k_u \notag\\
    &= \int_0^t x_0(s)\dd s +\int _{0}^{t} \int _{0}^{s} K_{\mu}(s,u)\,\mathrm{d} \hat{A}_{u}^{k} \,\mathrm{d}s  +\int_0^t K_\sigma(s) \int_0^{t-s} \dd \hat{M}^k_u \dd s\notag\\
    &= \int_0^t x_0(s)\dd s +\int _{0}^{t} \int _{0}^{s} K_{\mu}(s,u)\,\mathrm{d} \hat{A}_{u}^{k} \,\mathrm{d}s +\int_0^t K_\sigma(t-s) \hat{M}^k_s \dd s.\label{eq:limit2}
  \end{align}
  Since
  \begin{align*}
    \bigg\|\int_0^{\cdot}K_\sigma(\cdot-s)(\hat{M}^k_s-M_s)\dd s\bigg\|_{\infty}
	\leq \|\hat{M}^k-M\|_{\infty} \int_0^T |K_\sigma(T-s)|\dd s
    \leq C\|\hat{M}^k-M\|_{\infty}
  \end{align*}
  and $\hat{M}^k \to M$ as $k\to \infty$, $\tilde{\P}$-a.s, we obtain $\int_0^{\cdot} K_\sigma(t-s) \hat{M}_s^k \dd s\to \int_0^{\cdot} K_\sigma(t-s) M_s \dd s$ as $k\to \infty$, $\tilde{\P}$-a.s. The convergence of $\int _{0}^{t} \int _{0}^{s} K_{\mu}(s,u)\,\mathrm{d} \hat{A}_{u}^{k} \,\mathrm{d}s $ follows as in (a). Thus, taking the $\tilde{\P}$-a.s. limit of both sides of \eqref{eq:limit2} and then taking the derivative yields that \eqref{eq:X_MP} holds for $(X,Z)$, $\tilde{\P}$-a.s.
\end{proof}

\bibliography{literature}{}
\bibliographystyle{amsalpha}

\end{document}